\documentclass[a4paper,12pt]{book}
\usepackage[english]{babel}
\usepackage[utf8]{inputenc}
\usepackage{latexsym}
\usepackage{amssymb}
\usepackage{amsmath}
\usepackage{amsfonts}

\usepackage{graphicx}
\usepackage{epstopdf}
\usepackage{subfigure}
\usepackage[nottoc]{tocbibind}
\usepackage{float}
\usepackage{color}

\usepackage{multind}

\numberwithin{equation}{section}

\newcommand{\file}{\jobname}
\makeindex{\file}
\makeindex{\file-1}

\newtheorem{theorem}{Theorem}[section]
\newtheorem{proposition}[theorem]{Proposition}
\newtheorem{corollary}[theorem]{Corollary}
\newtheorem{lemma}[theorem]{Lemma}
\newtheorem{definition}[theorem]{Definition}

\newenvironment{proof}{\begin{trivlist} \item[] \textbf{Proof.}}{\quad \rule{2mm}{2mm} \end{trivlist}}
\newenvironment{proof*}[1]{\begin{trivlist} \item[] \textbf{Proof of #1.} }{\quad \rule{2mm}{2mm} \end{trivlist}}

\author{Zura Dvalashvili and Nato Nadirashvili\\
The University of Georgia \\ School of Science and Technology \\ 77a Merab Kostava St, Tbilisi, 0128, Georgia \\
		e-mail: zurabdvalashvili@gmail.com \\ 	
		e-mail: nato.nadirashvili@gmail.com \\
\and
Supervisor:\\
 George Tephnadze\\
The University of Georgia \\ School of Science and Technology \\ 77a Merab Kostava St, Tbilisi, 0128, Georgia \\
e-mail: g.tephnadze@ug.edu.ge }

\title{\huge Master Thesis \\
	 On the almost everywhere and norm convergences of N\"orlund means with respect to Vilenkin systems}

\date{2021}

\begin{document}

\maketitle

\newpage

\thispagestyle{empty}
\begin{center}
	{\scriptsize This page intentionally left blank}
\end{center}

\newpage




\pagenumbering{roman}
\tableofcontents

\newpage

\section*{Preface}
\markboth{PREFACE}{}

\addcontentsline{toc}{section}{Preface}

\qquad Unlike the classical theory of Fourier series which deals with decomposition of a function into sinusoidal waves the Vilenkin (Walsh) functions are rectangular waves. The development of the theory of Vilenkin-Fourier series has been strongly influenced by the classical theory of trigonometric series but there are a lot of differences also. The aim of my master thesis is to discuss, develop and apply the newest developments of this fascinating theory connected to modern harmonic analysis. In particular, we investigate N\"orlund  means but only in the case  when their coefficients are monotone and prove convergence in Lebesgue and Vilenkin-Lebesgue points. Since almost everywhere points are Lebesgue and Vilenkin-Lebesgue points for any integrable functions we obtain almost everywhere convergence of such summability methods.

\newpage

\section*{Key words}
\markboth{Key words}{}

\addcontentsline{toc}{section}{Key words}

\textbf{Key words:} Vilenkin systems, N\"orlund  means,  almost everywhere convergence, Lebesgue points, Vilenkin-Lebesgue points, convergence, approximation.

\pagenumbering{arabic} 

\chapter{Introduction}
\section{Vilenkin Groups and Functions}\label{s0,2}
\markright{Vilenkin Groups and Functions}

\ \ \ \ \ \ Denote by 
$\mathbb{N}_+$ the set of the positive integers,
$\mathbb{N}:=\mathbb{N}_+\cup \{0\}$, 
\index{\file-1}{$\mathbb{N}_+$} \index{\file}{positive integers}
\index{\file-1}{$\mathbb{N}$}
$\mathbb{Z}$ \index{\file-1}{$\mathbb{Z}$}
\index{\file}{integers} the set of the integers,
$\mathbb{R}$ \index{\file-1}{$\mathbb{R}$}
\index{\file}{real numbers}
the real numbers,
$\mathbb{R}_+$ \index{\file-1}{$\mathbb{R}_+$}
\index{\file}{positive real numbers}
the positive real numbers,
$\mathbb{C}$ \index{\file-1}{$\mathbb{C}$}
\index{\file}{complex numbers}
the complex numbers.
Let $m:=(m_0,m_1,\ldots)$ \index{\file-1}{$m_k$} be a sequence of positive
integers not less than 2. Denote by \index{\file-1}{$Z_{m_k}$}
\begin{equation*}
Z_{m_k}:=\{0,1,\ldots,m_k-1\}
\end{equation*}%
the additive group of integers modulo $m_k$ \index{\file}{additive group of integers modulo $m_k$}.

Define the group $G_m$ \index{\file-1}{$G_m$} as the complete direct product of the groups $Z_{m_k}$ with the product of the discrete topologies of $Z_{m_k}$.

The direct product $\mu $ of the measures
\begin{equation*}
\mu_k\left(j\right):=1/m_k\qquad(j\in Z_{m_k})
\end{equation*}%
is the Haar measure on $G_m$ with $\mu\left(G_m\right)=1.$

If $\sup_{n\in\mathbb{N}}m_n<\infty$, then we call $G_m$ a bounded Vilenkin \index{A}{Vilenkin} group.
If the generating sequence $m$ is not bounded, then $G_m$ is said to be an unbounded Vilenkin group.

In this book we discuss only bounded Vilenkin groups,\index{\file}{bounded Vilenkin groups} i.e. the case when $\sup_{n\in\mathbb{N}}m_n<\infty.$

The elements of $G_m$ are represented by sequences

\begin{equation*}
x:=\left(x_0,x_1,\ldots,x_j,\ldots\right)\qquad\left(x_j\in Z_{m_j}\right).
\end{equation*}

It is easy to give a base for the neighborhoods of $G_m:$  \index{\file-1}{$I_0\left(x\right)$}  \index{\file-1}{$I_n\left(x\right)$}

\begin{eqnarray*}
	I_0\left(x\right)&:&=G_m,\\
	I_n(x)&:&=\{y\in G_m\mid y_0=x_0,\ldots
	,y_{n-1}=x_{n-1}\} \qquad\left(x\in G_m,\ \ n\in\mathbb{N}\right).
\end{eqnarray*}

We call subsets $I_n(x)\subset G_m$ Vilenkin intervals. Let  \index{\file-1}{$ e_n$} \index{\file}{Vilenkin polynomials}

\begin{equation*}
e_n:=\left(0,\ldots,0,x_n=1,0,\ldots\right)\in G_m\qquad \left(n\in\mathbb{N}\right).
\end{equation*}

If we define $I_n:=I_n\left(0\right),$ \index{\file-1}{$I_n$} for  $n\in\mathbb{N}$ and $\overline{I_n}:=G_m \ \ \backslash $ $I_n,$ \index{\file-1}{$\overline{I_n} $} then

\begin{equation} \label{1.1}
\overline{I_N}=\overset{N-1}{\underset{s=0}{\bigcup}}I_s\backslash
I_{s+1}=\left(\overset{N-2}{\underset{k=0}{\bigcup}}\overset{N-1}
{\underset{l=k+1}{\bigcup}}I_N^{k,l}\right)\bigcup
\left(\underset{k=1}{\bigcup\limits^{N-1}}I_N^{k,N}\right),
\end{equation}
where \index{\file-1}{$I_N^{k,l}$}

\begin{equation*}
I_N^{k,l}:=\left\{ \begin{array}{l}I_N(0,\ldots,0,x_k\neq 0,0,...,0,x_l\neq 0,x_{l+1},\ldots ,x_{N-1},\ldots),\\
\text{for} \qquad k<l<N,\\
I_N(0,\ldots,0,x_k\neq 0,x_{k+1}=0,\ldots,x_{N-1}=0,x_N,\ldots ), \\
\text{for } \qquad l=N. \end{array}\right.
\end{equation*}

If we define the so-called generalized number system based on $m$  in the
following way : \index{\file-1}{$M_k$}

\begin{equation*}
M_0:=1, \ \  M_{k+1}:=m_kM_k \ \  (k\in\mathbb{N}),
\end{equation*}
then every $n\in\mathbb{N}$ can be uniquely expressed as

\begin{equation*}
n=\sum_{j=0}^{\infty}n_jM_j,
\end{equation*}
where $n_j\in Z_{m_j} \ \ (j\in\mathbb{N}_+)$ and only a finite number of $n_j^{^{\prime}}$`s differ from zero.

The Vilenkin group can be metrizable with the following metric:

\begin{equation*}
\rho\left(x,y\right):=\left\vert x-y\right\vert:=\sum_{k=0}^{\infty}
\frac{\left\vert x_k-y_k\right\vert}{M_{k+1}},\qquad\left( x,y\in G_m\right).
\end{equation*}

For the natural numbers $n=\sum_{j=1}^{\infty}n_jM_j$ and $k=\sum_{j=1}^{\infty}k_jM_j$ we define

\begin{equation*}
n\widehat{+}k:=\sum_{i=0}^{\infty}{\left(n_i\oplus k_i\right)}{M_{i+1}}
\end{equation*}
and \index{\file-1}{$\widehat{-}$}

\begin{equation*}
n\widehat{-}k:=\sum_{i=0}^{\infty}{\left(n_i\ominus k_i\right)} {M_{i+1}},
\end{equation*}
where

\begin{equation*}
a_i\oplus b_i:=(a_i+b_i)\text{mod}m_i, \qquad a_i, b_i \in Z_{m_i}
\end{equation*}
and $\ominus $ \index{\file-1}{$\ominus$} is the inverse operation for  \index{\file-1}{$\oplus$} $ \oplus $.

Next, we introduce on $G_m$ an orthonormal system, which is called
Vilenkin system (see \cite{Vi,Vi1,Vi2}).

At first, we define the complex-valued function
$r_k\left(x\right):G_m\rightarrow\mathbb{C},$
the generalized Rademacher functions, \index{\file}{generalized Rademacher functions} by \index{\file-1}{$r_k\left(x\right)$}

\begin{equation}\label{rad0}
r_k\left(x\right):=\exp\left(2\pi \imath x_{k}/m_{k}\right),\ \ \left(
\imath^{2}=-1, \ \ x\in G_{m},\ \ k\in\mathbb{N}\right).
\end{equation}

Now, define the Vilenkin systems \index{\file}{Vilenkin systems}
$\psi:=(\psi_n:n\in\mathbb{N})$ on $G_m$ by: \index{\file-1}{$\psi_n(x)$}

\begin{equation}\label{Vilenkin}
\psi_n(x):=\prod\limits_{k=0}^{\infty}r_k^{n_k}\left(x\right), \ \ \left(n\in\mathbb{N}\right).
\end{equation}

Specifically, we call this system the Walsh-Paley system \index{\file}{Walsh-Paley system} \index{A}{Walsh-Paley} when $m\equiv 2.$

\begin{proposition}(see \cite{AVD}) \label{vilprop}
	Let $n\in\mathbb{N}.$ Then
	
	\begin{eqnarray}
	\left\vert \psi _{n}\left( x\right) \right\vert &=&1,\text{ }\notag \\
	\psi _{n}\left( x-y\right) &=&\psi _{n}\left( x\right) \text{\ }%
\overline{\psi }_{n}\left( y\right) .  \notag 
	\end{eqnarray}
\end{proposition}

The direct product $\mu $ of the measures

\begin{equation*}
\mu _{k}\left( \{j\}\right) :=1/m_{k}\text{ \qquad }(j\in Z_{m_{k}})
\end{equation*}%
is the Haar measure on $G_{m_{\text{ }}}$with $\mu \left( G_{m}\right) =1.$
Translation of a subset $I_{n}(x)\in G_{m}$ by $y$ is defined by
$\tau_y\left( I_{n}(x)\right) =\{I_{n}(x)+y\}.$
Since $\mu$ is a product measure we get that
$\mu \left( I_{n}(x)\right)=1/{M_n}$
and

$$\mu \left( \tau _{y}\left(
I_{n}(x)\right) \right) =\mu \left( I_{n}(x+y)\right)=1/M_n$$
for all $y\in G_{m}.$ Hence,
$$\mu\left(\tau_y\left(I_n(x)\right)\right)=\mu \left(I_n(x)\right).$$

\begin{proposition} \label{prop1.5} Let $n,k\in \mathbb{N}.$ Then
	
	\begin{equation*}
	\int_{G_{m}}\psi _{n}d\mu \text{\thinspace }=\left\{
	\begin{array}{ll}
	1\text{ \ \ \ \ }\, & n=0, \\
	0 & n\neq 0.%
	\end{array}%
	\right.
	\end{equation*}
	
	Moreover, the Vilenkin systems are orthonormal, that is,
	
	\begin{equation*}
	\int_{G_{m}}\psi _{n}\overline{\psi _{k}}d\mu \text{\thinspace }=\left\{
	\begin{array}{ll}
	1\text{ \ \ \ \ }\, & n=k, \\
	0 & n\neq k.%
	\end{array}%
	\right.
	\end{equation*}
\end{proposition}

\section{$L_p$ and weak-$L_p$ Spaces }\label{s0,5}
\markright{Lebesgue Spaces}

\ \ \ \ By a Vilenkin polynomial we mean a finite linear
combination of Vilenkin functions. We denote the collection of Vilenkin polynomials by $\mathcal{P}$. \index{\file-1}{$\mathcal{P}$} \index{\file}{Vilenkin polynomials}

Let $L^0(G_m)$ represent the collection of functions which are almost everywhere limits with respect to a measure $\mu$ of sequences in $\mathcal{P}$.

For $0<p<\infty$ let $L^p(G_m)$ represent the collection of $f\in L^0(G_m)$ such that \index{A}{Lebesgue} \index{\file}{Lebesgue spaces}   \index{\file-1}{$L^p(G_m)$}
\begin{equation*}
\left\Vert f\right\Vert_p:=\left(\int_{G_m}\left\vert f\right\vert
^pd\mu\right)^{1/p}
\end{equation*}
is finite.

Denote by $L^{\infty }(G_{m})$ the space of all  $f\in L^0(G_m)$ for which

\begin{equation*}
\left\Vert f\right\Vert _{\infty }:=\inf \left\{ C>0:\mu \left\{ x\in G_{m}:\left\vert f\right\vert >C\right\} =0\right\} <+\infty .
\end{equation*}

The space $C(G_{m})$ consist all continuous function for which
\begin{equation*}
\left\Vert f\right\Vert _{C}:=\underset{x\in G_{m}}{\sup }\left\vert
f(x)\right\vert<c<\infty.
\end{equation*}

\begin{proposition}(see \cite{Tor1})
	\label{Remark1.1.4} Well-known Minkowski's integral inequality  is given by 
	
	\begin{equation*}
	\left\Vert \int_{G_m}f\left( \cdot ,t\right) dt\right\Vert _{p}\leq\int_{G_m}\left\Vert f\left( \cdot ,t\right) \right\Vert _{p}dt, \ \ \ \text{for all} \ \ \ p\geq 1.
	\end{equation*}
\end{proposition}

The convolution of two functions
$f,g\in L^{1}(G_m)$ is defined by

\begin{equation*}
\left( f\ast g\right) \left( x\right) :=\int_{G_m}f\left( x-t\right)
g\left( t\right) dt\text{ \ \ }\left( x\in G_m\right).
\end{equation*}%
It is easy to see that

\begin{equation*}
\left( f\ast g\right) \left( x\right) =\int_{G_m}f\left( t\right)
g\left( x-t\right) dt\text{ \ \ }\left( x\in G_m\right).
\end{equation*}

\begin{proposition}\label{Theorem 1.1.6} 
	Let $f\in L^{r}\left( G_m\right) ,$ $g\in
	L^{1}\left(  G_m\right) $ and $1\leq r<\infty .$ Then $f\ast g\in
	L^{r}\left(  G_m\right) $ and
	
	\begin{equation*}
	\left\Vert f\ast g\right\Vert _{r}\leq \left\Vert f\right\Vert_{r}\left\Vert g\right\Vert _{1}.
	\end{equation*}
\end{proposition}

First we present the following very important proposition:

\begin{proposition}\label{remarkae}
	Since the Vilenkin function $\psi _{m}$ is constant on $I_{n}(x)$ for every $x\in G_{m}$ and $0\leq m<M_{n},$ it is clear that each Vilenkin function is a complex-valued step function, that is, it is a finite linear combination of the characteristic functions.
	On the other hand, notice that, by Lemma \ref{dn2.3} (Paley's Lemma), it yields that
	
	\begin{equation*}
	\chi_{I_n(t)} \left(x\right) =\frac{1}{M_n}\sum_{j=0}^{M_n-1}\psi_j\left( x-t\right) ,\text{ \ \ }x\in I_n(t),
	\end{equation*}
	for each $x,t\in $ $G_{m}$ and $n\in \mathbb{N}$. Thus each step function is a Vilenkin polynomial. Consequently, we obtain that the collection of step functions coincides with a collection of Vilenkin polynomials $\mathcal{P}$. 
	
	Since the Lebesgue measure is regular it follows  that given $f\in L^{1}$ there exist Vilenkin polynomials
	
	$$P_1,P_2..., \ \ \ \text{ such that } \ \  \ P_{n}\rightarrow f \ \ \ \text{  a.e., as }  \ \ \ n\rightarrow \infty .$$ 
	Moreover, any $f\in L^p(G_m)$ can be written in the form $f = g - h$ where the functions $g, \ h$ are almost everywhere limits of increasing sequences of non-negative Vilenkin polynomials. In particular, $\mathcal{P}$ is dense in the space $L^p,$  for all $p\geq 1.$
\end{proposition}

The space $weak-L^p\left(G_m\right)$ \index{\file-1}{$weak-L^p\left( G_m\right) $} consists of all measurable functions $f$, for which
\begin{equation*}
\left\Vert f\right\Vert_{weak-L^p}:=\underset{y>0}{\sup}y
\{\mu\left(f>y\right)\}^{1/p}<+\infty.
\end{equation*}

\begin{proposition}\label{proposition 1.1.2} (see \cite{Tor1}) If $0<p\leq \infty ,$ then $L^{p}\left(G_m\right) \subset \text{weak}-L^p(G_m)$ and
	\begin{equation*}
	\left\Vert f\right\Vert_{\text{weak}-L^p}\leq \left\Vert f\right\Vert_{p}.
	\end{equation*}
\end{proposition}

\begin{proof}
	It is easy to see that
	
	\begin{equation*}
	\int_{G_m}\left\vert f\left( x\right) \right\vert^{p}dx\geq\int_{\left\{ x:\left\vert f\left(x\right) \right\vert >y \right\} }\left\vert f\left(x\right)\right\vert^p dx\geq y^p \mu\left(\left\vert f\right\vert>y \right),
	\end{equation*}
	which proves the proposition.
\end{proof}

An operator $T$ which maps a linear space of measurable functions on $G_m$ in the collection
of measurable functions on $G_m$ is called sublinear if
$$\vert T(f + g)\vert \leq \vert T(f)\vert+\vert T( g)\vert
\text{ \ \ a.e. on \ \ } G_m \text{ \ \ and \ \ }
\vert T(\alpha f)\vert=\vert \alpha \vert \vert T( f)\vert$$
for all scalars $\alpha$ and all $f$ in the domain of $T.$

\section{Dirichlet and Vilenkin-Fej\'er Kernels }\label{s0,6}
\markright{Dirichlet Kernels}

\ \ \ \ \ If $f\in L^1\left( G_m\right) $ we can define the
Fourier coefficients, \index{\file}{Fourier coefficients} the partial sums  of Vilenkin-Fourier series, \index{\file}{partial sums of Vilenkin-Fourier series} the Dirichlet
kernels with respect to Vilenkin systems \index{A}{Dirichlet} \index{\file}{Dirichlet
	kernels with respect to Vilenkin systems} in the usual manner:
\index{\file-1}{$\widehat{f}\left( n\right)$}
\index{\file-1}{$S_{n}f$}
\index{\file-1}{$D_{n}$}
\begin{eqnarray*}
	\widehat{f}\left(n\right)&:=&\int_{G_m}f\overline{\psi}_nd\mu,\qquad \left(n\in \mathbb{N}\right), \\
	S_nf&:=&\sum_{k=0}^{n-1}\widehat{f}\left(k\right)\psi_k, \qquad\left(n\in\mathbb{N}_+\right), \\
	D_n&:=&\sum_{k=0}^{n-1}\psi_{k},\qquad\left(n\in\mathbb{N}_+\right),
\end{eqnarray*}
respectively. It is easy to see that
\begin{eqnarray*}
	S_nf\left(x\right)&=&\int_{G_m}f\left(t\right)\sum_{k=0}^{n-1}\psi
	_k\left(x-t\right)d\mu \left(t\right) \\
	&=&\int_{G_m}f\left( t\right) D_n\left(x-t\right)d\mu\left(t\right)
	=\left(f\ast D_n\right)\left(x\right).
\end{eqnarray*}

The next well-known identities with respect to Dirichlet kernels (see Lemmas \ref{dn1} and \ref{dn2.1}, Lemma \ref{dn2.3}) will be used many times in the proofs of our main results:

\begin{lemma} (see \cite{AVD}) \label{dn}
	\label{dn1}Let $n\in\mathbb{N}.$ Then
	\begin{equation} \label{dn21}
	D_{j+M_n}=D_{M_n}+\psi_{M_n}D_j=D_{M_n}+r_nD_j, \qquad j\leq\left( m_n-1\right)M_n
	\end{equation}
	and
	\begin{eqnarray} \label{dn22}
	D_{M_n-j}(x)&=&D_{M_n}(x)-\overline{\psi}_{M_n-1}(-x)D_j(-x)\\ \notag
	&=&D_{M_n}(x)-\psi_{M_n-1}(x)\overline{D}_j(x), \qquad
	j<M_n.
	\end{eqnarray}
\end{lemma}

\begin{lemma} (see \cite{AVD}) \label{dn2.1}
	Let $n\in\mathbb{N}$ and $1\leq s_n\leq m_n-1.$ Then \index{\file-1}{$D_{s_nM_n}$}
	\begin{equation} \label{9dn}
	D_{s_nM_n}=D_{M_n}\sum_{k=0}^{s_n-1}\psi_{kM_n}=D_{M_n}\sum_{k=0}^{s_n-1}r_n^k
	\end{equation}
	and
	\begin{equation} \label{2dn}
	D_n=\psi_n\left(\sum_{j=0}^{\infty}D_{M_j}\sum_{k=m_j-n_j}^{m_j-1}r_j^k\right),
	\end{equation}
	for $n=\sum_{i=0}^{\infty}n_iM_i.$
\end{lemma}

\begin{lemma} (see \cite{AVD} and \cite{gol}) \textbf{(Paley's Lemma)}\label{dn2.3}
	\index{\file-1}{$D_{M_n}$}
	Let $n\in \mathbb{N}.$ Then%
	\begin{equation*}
	D_{M_n}\left(x\right)=\left\{ \begin{array}{ll} M_n, & x\in I_n, \\
	0, & x\notin I_n. \end{array} \right.
	\end{equation*}
\end{lemma}

We also need the following estimate:
\begin{lemma}(see \cite{AVD})\label{dn2.6int1}Let $n\in \mathbb{N}$. Then
	
	\begin{equation*}
	\Vert D_{M_n}\Vert_1=1.
	\end{equation*}
\end{lemma}

\begin{lemma} (see \cite{AVD} and \cite{gol}) \label{dn2.6}
	Let $x\in I_s\backslash I_{s+1}, \ \  s=0,...,N-1.$  Then
	
	\begin{equation*}
	\left\vert D_n\left(x\right)\right\vert \leq cM_s
	\end{equation*}
	and
	\begin{equation*}
	\int_{I_N}\left\vert D_n\left(x-t\right)\right\vert d\mu\left(
	t\right)\leq\frac{cM_s}{M_N},
	\end{equation*}
	where $c$ is an absolute constant.
\end{lemma}

It is obvious that
\begin{eqnarray*}
\sigma_nf\left(x\right) &=&\frac{1}{n}\overset{n-1}{\underset{k=0}{\sum}}
\left(D_k\ast f\right)\left(x\right) \\
&=&\left(f\ast K_n\right)\left(x\right)=\int_{G_m}f\left(t\right)
K_n\left(x-t\right)d\mu \left(t\right),
\end{eqnarray*}
where $ K_n $ are the so called Fej\'er kernels: \index{\file}{Fej\'er kernels} \index{\file-1}{$K_n$}

\begin{eqnarray*}
K_n:=\frac{1}{n}\overset{n-1}{\underset{k=0}{\sum}}D_k.
\end{eqnarray*}

Using Abel transformation we get another representation of Fej\'er means

$$\sigma_nf\left(x\right)=\overset{n-2}{\underset{k=0}{\sum}} \left(1-\frac{k}{n}\right) \widehat{f}\left( k\right)\psi_k\left(x\right)$$

We frequently use the following well-known result:

\begin{lemma} (see \cite{gat}) \label{lemma2}
 Let $n>t,$ $t,n\in \mathbb{N}.$ Then \index{\file-1}{$K_{M_n}$}
 
\begin{equation*}
K_{M_n}\left(x\right)=\left\{ \begin{array}{ll}
\frac{M_t}{1-r_t\left(x\right) },& x\in I_t\backslash I_{t+1},\qquad x-x_te_t\in I_n, \\
\frac{M_n+1}{2}, & x\in I_n, \\
0, & \text{otherwise. } \end{array} \right.
\end{equation*}
\end{lemma}

The proof of the next lemma can easily be done by using Lemma  \ref{lemma2}:

\begin{lemma} \label{lemma222}
	Let $n\in \mathbb{N}$ and $x\in I_N^{k,l},$ where $k<l.$
	Then
	
	\begin{equation} \label{star1}
	K_{M_n}\left(x\right)=0,\text{ \ if \ } n>l,
	\end{equation}
	
	\begin{equation} \label{star2}
	\left\vert K_{M_n}\left(x\right)\right\vert\leq cM_k,
	\end{equation}
	and

	\begin{equation} \label{fn50}
	\left\vert K_{M_n}(x)\right\vert\leq c\sum_{s=0}^ {n}M_s\sum_{r_s=1}^{m_s-1}\chi_{I_n(x-r_se_s)}
	\end{equation}%
	Moreover,
	
	\begin{equation} \label{star3}
	\int_{G_m}\left\vert K_{M_n}\right\vert d\mu\leq c<\infty,
	\end{equation}%
	where $c$ is an absolute constant.
\end{lemma}

We also need the following useful result:

\begin{lemma}(see \cite{AVD} and \cite{gol}) \label{lemma6kn}
Let $t,s_n, \ \ n\in \mathbb{N},$ and $1\leq s_n\leq m_n-1$. Then \index{\file-1}{$K_{s_nM_n}$}

\begin{equation} \label{mag1}
s_nM_nK_{s_nM_n}=\sum_{l=0}^{s_n-1}\left(\sum_{i=0}^{l-1}r_n^i\right)M_nD_{M_n}+\left(\sum_{l=0}^{s_n-1}r_n^{l}\right)M_nK_{M_n}
\end{equation}%
\end{lemma}

The next equality for Fej\'er kernels is very important for our further investigations:

\begin{lemma}(see \cite{AVD} and \cite{gol}) \label{lemma4}
Let $n=\sum_{i=1}^rs_{n_i}M_{n_i}$, where
$n_1>n_2>\dots>n_r\geq 0$ and $1\leq s_{n_i}<m_{n_i}$ \ for all
$1\leq i\leq r$ as well as $n^{(k)}=n-\sum_{i=1}^ks_{n_i}M_{n_i}$,
where $0<k\leq r$. Then

\begin{eqnarray}\label{starkn}
nK_n&=&\sum_{k=1}^r\left(\prod_{j=1}^{k-1}r_{n_j}^{s_{n_j}}\right)
s_{n_k}M_{n_k}K_{s_{n_k}M_{n_k}}\\ \notag
&+&\sum_{k=1}^{r-1}\left(\prod_{j=1}^{k-1}r_{n_j}^{s_{n_j}}\right) n^{(k)}D_{s_{n_k}M_{n_k}}.
\end{eqnarray}
\end{lemma}

We will also frequently use the next estimation of the Fej\'{e}r kernels:

\begin{corollary}\label{lemma7kn0}
Let $n\in \mathbb{N}.$ Then

\begin{equation} \label{fn5}
n\left\vert K_n\right\vert\leq c\sum_{l=\left\langle n\right\rangle}^ {\left\vert n\right\vert}M_l\left\vert K_{M_l}\right\vert\leq
c\sum_{l=0}^{\left\vert n\right\vert }M_l\left\vert K_{M_l}\right\vert
\end{equation}%
where $c$ is an absolute constant.
\end{corollary}

\begin{lemma}(see \cite{AVD} and \cite{gol}) \label{lemma7kn}
Let $n\in \mathbb{N}.$ Then, for any $n,N\in \mathbb{N_+}$, we have that
\begin{eqnarray} \label{fn40}
&&\int_{G_m} K_n (x)d\mu(x)=1,\\
&& \label{fn4}
\sup_{n\in\mathbb{N}}\int_{G_m}\left\vert K_n(x)\right\vert d\mu(x)\leq c<\infty,\\
&& \label{fn400}
\sup_{n\in\mathbb{N}}\int_{G_m \backslash I_N}\left\vert K_n(x)\right\vert d\mu (x)\rightarrow  0, \ \ \text{as} \ \ n\rightarrow  \infty, 
\end{eqnarray}
where $c$ is an absolute constant.
\end{lemma}

\begin{lemma}\label{lemma0nnT12} (see \cite{BPTW,bpt1,bpt2,ptw}) Let $\{q_k:k\in\mathbb{N}\}$ be a sequence of non-decreasing numbers, satisfying the condition \eqref{fn01}. Then for any  $ n, N\in \mathbb{N_+},$
	\begin{eqnarray} \label{1.711}
	&&\int_{G_m} F^{-1}_n(x) d\mu (x)=1, \\
	&&\sup_{n\in\mathbb{N}}\int_{G_m}\left\vert F^{-1}_n(x)\right\vert d\mu(x)\leq c<\infty,\label{1.721} \\
	&&\sup_{n\in\mathbb{N}}\int_{G_m \backslash I_N}\left\vert F^{-1}_n(x)\right\vert d\mu (x)\rightarrow  0, \ \ \text{as} \ \ n\rightarrow  \infty, \label{1.731}
	\end{eqnarray}
	where $c$ is an absolute constant.
\end{lemma}

\chapter{N\"orlund Means of Vilenkin-Fourier series in Lebesgue Spaces}\label{Ch3}
\markboth{N\"orlund Means of Vilenkin-Fourier series in Lebesgue Spaces}{}

\section{Introduction}\label{s3.1}
\markright{Introduction}

\ \ \ \ In the literature, a point $x\in G_m$ is called a Lebesgue point of   $f\in L^{1}\left( G_m\right),$ if

\[\lim_{n\rightarrow \infty }M_n\int_{I_{n}(x)}f\left(t\right) dt=f\left( x\right) \ \ \ \ \ a.e.\ x\in G_m.\]

It is well-known (for details see \cite{Tor1} and \cite{sws}) that if  $f\in L^{1}\left( G_m\right)$ then almost every point is a Lebesgue point and the following important result holds true:

\begin{proposition}\label{Snae1}
	Let $f\in L^1(G_m)$. Then, for all Lebesgue points $x$,
	
	\begin{equation*}
	\underset{n\rightarrow \infty }{\lim }S _{M_n}f(x)=f(x)
	\end{equation*}	
\end{proposition}

In the literature there are notations of $n$-th N\"orlund $L_n$ and Riesz $R_n$ logarithmic means \index{\file}{N\"orlund logarithmic mean}  \index{\file}{Riesz logarithmic mean}
defined by
\index{\file-1}{$L_nf$}
\index{\file-1}{$R_nf$}
\begin{eqnarray*}
	L_nf&:=&\frac{1}{l_n}\sum_{k=0}^{n-1}\frac{S_kf}{n-k},\\
	R_nf&:=&\frac{1}{l_n}\sum_{k=1}^{n}\frac{S_kf}{k},
\end{eqnarray*}
respectively, where \index{\file-1}{$l_n$}

\begin{equation*}
l_n:=\sum_{k=1}^{n}\frac{1}{k}.
\end{equation*}

It is known that the N\"orlund logarithmic mean has better approximation properties then the partial sums and that the Riesz logarithmic means is better than Fej\'{e}r means in the same sense, but they have much more similar properties with them. In \cite{Ga2} G\'{a}t and Goginava proved some
convergence and divergence properties of the N\"orlund logarithmic means of
functions in the class of continuous functions and in the Lebesgue space $L^1.$ Moreover, G\'{a}t and Goginava \cite{Ga3} proved that for each measurable function satisfying

$$\phi \left( u\right) =o\left( u\log ^{1/2}u\right) ,\text{ \  as \ }u\rightarrow \infty,$$
there exists an integrable function $f$ such that
\begin{equation*}
\int_{G_{m}}\phi \left( \left\vert f\left( x\right) \right\vert \right) d\mu
\,\left( x\right) <\infty
\end{equation*}%
and that there exists a set with positive measure such that the N\"orlund logarithmic means of the function diverges on this set. It follows that  weak-(1,1) type inequality does not hold for the maximal operator of N\"orlund logarithmic means:
\index{\file}{maximal operator of N\"orlund logarithmic means}

\begin{eqnarray*}
	L^{\ast}f&:=&\sup_{n\in\mathbb{N}}\left\vert L_nf\right\vert
\end{eqnarray*}
but there exists an absolute constant $C_p$ such that

\begin{equation*}
\left\Vert L^{\ast }f\right\Vert _{p}\leq C_{p}\left\Vert f\right\Vert _{p},%
\text{ \ when \ }f\in L^{p},\text{ \ }p>1.
\end{equation*}
Moreover (for details see \cite{BSPT}), if we consider the following restricted maximal operator $	\widetilde{L}_{\#}^{\ast}f,$ defined by \index{\file}{restricted maximal operator of partial sums}
\begin{eqnarray*}
	\widetilde{L}_{\#}^{\ast}f&:=&\sup_{n\in\mathbb{N}}\left\vert L_{M_n}f\right\vert, \ \ \ (M_{k}:=m_0...m_{k-1}, \ \ \  k=0,1...),
\end{eqnarray*}
then
\begin{eqnarray*}
	y \mu\left\{ \widetilde{L}_{\#}^{\ast}f>y \right\} \leq c\left\Vert f\right\Vert_{1}, \ \ \ f\in L^1(G_m), \ \ y>0.
\end{eqnarray*}
Hence, if $f\in L^1(G_m),$ then
$$L_{M_n}f\to f, \ \ \text{a.e. on } \ \ G_m.$$

Fejér's theorem shows that (see e.g books \cite{Fejer3} and \cite{Fejer}) if one replaces ordinary summation by Cesàro summation $\sigma_n$ defined by
\index{\file-1}{$\sigma_nf$}

\begin{equation*}
\sigma_nf:=\frac{1}{n}\sum_{k=1}^nS_kf,
\end{equation*}
then the F\'ejer means of Fourier series of any integrable function converges a.e on $G_m$ to the function.   

Goginava and Gogoladze  \cite{GG1} introduced the operator $W_A$ defined by

\begin{eqnarray*}
	W_Af(x):= \sum_{s=0}^{A-1}M_s\sum_{r_s=1}^{m_s-1}\int_{I_A(x-r_se_s)}\left\vert f(t)-f\left( x\right)\right\vert d\mu(t).
\end{eqnarray*} 
and define a Vilenkin-Lebesgue point of function $f \in L^1 (G_m),$ as a point for which

\begin{eqnarray*}
	\lim_{A \rightarrow \infty} W_Af (x)=0
\end{eqnarray*}
Moreover, they also proved that the following result is true:
\begin{proposition} \label{villebfej}
	Let $f\in L^1(G_m)$. Then
	
	\begin{equation*}
	\underset{n\rightarrow \infty }{\lim }\sigma _{n}f(x)=f(x)
	\end{equation*}
	for all Vilenkin-Lebesgue points of $f$.
\end{proposition}

Moreover, the following is true:

\begin{proposition} \label{Pointwize}
	Let $x\in G_m$ and $f\in L^1 (G_m)$ is continuous at the point $x$. Then 
	
	$$\underset{m\to \infty }{\mathop{\lim }}\,{{\sigma }_{m}}f\left( x \right)=f\left( x \right).$$
\end{proposition}	

If we consider the maximal operator of F\'ejer means $\sigma^{\ast}$ defined by: \index{\file}{maximal operator of F\'ejer means}  \index{\file-1}{$\sigma^{\ast}f$}

\begin{eqnarray*}
	\sigma ^{\ast}f&:=&\sup_{n\in\mathbb{N}}\left\vert \sigma_nf\right\vert
\end{eqnarray*}
then

\begin{eqnarray*}
	y\mu\left\{ \sigma ^{\ast}f>y \right\} \leq c\left\Vert f\right\Vert_{1}, \ \ \ f\in L^1(G_m), \ \ y>0.
\end{eqnarray*}
This result can be found in Zygmund \cite{13z} for trigonometric series, in Schipp \cite{Sc} for Walsh series and in P\'al, Simon \cite{PS} for bounded Vilenkin series.

The boundedness does not hold from Lebesgue space $L^1(G_m)$ to the space $L^1(G_m)$. On the other hand, if we consider restricted maximal operator $\widetilde{\sigma}_{\#}^{\ast}$ of F\'ejer means defined by \index{\file}{restricted maximal operator of F\'ejer means}
\index{\file-1}{$\widetilde{\sigma}_{\#}^{\ast}f$}

\begin{eqnarray*}
	\widetilde{\sigma}_{\#}^{\ast}f&:=&\sup_{n\in\mathbb{N}}\left\vert \sigma_{M_n}f\right\vert
\end{eqnarray*}
then there exists a function $f\in L^1(G_m)$ such that

$$\Vert\widetilde{\sigma}_{\#}^{\ast}f\Vert_1=\infty.$$

In the one-dimensional case Yano \cite{Yano} proved that

\begin{equation*}
\left\Vert \sigma _{n}f-f\right\Vert _{p}\rightarrow 0,\text{ \ \ \ as \ \ \ \ }n\rightarrow \infty ,\text{ \ }(f\in L^p(G_m),\text{ \ }1\leq p\leq \infty ).
\end{equation*}%
However (see \cite{JOO, sws}) the rate of convergence can not be better then
$O\left( n^{-1}\right) $ $\left( n\rightarrow \infty \right) $ for
non-constant functions, i.e., if $f\in L^{p},$ $1\leq p\leq \infty $ and

\begin{equation*}
\left\Vert \sigma _{M_{n}}f-f\right\Vert _{p}=o\left( \frac{1}{M_{n}}\right)
,\text{ \ as \ \ }n\rightarrow \infty ,
\end{equation*}%
then \textit{\ }$f$ \ is a constant function.

It is also known that (see e.g the  books \cite{AVD} and
\cite{sws}) for any $1\leq p\leq \infty$ and $n\in \mathbb{N}$ we have the following estimate

\begin{equation*}
\left\Vert \sigma_nf-f\right\Vert_p\leq c_p \omega_p\left( \frac{1}{M_N},f\right) +c_p\sum_{s=0}^{N-1}\frac{M_s}{M_N}\omega_p\left( \frac{1}{M_s},f\right).
\end{equation*}
where $\omega_p\left( \frac{1}{M_n},f\right)$ is the modulus of continuity of function $f\in L^p:$

\begin{equation*}
\omega_p\left(\frac{1}{M_n},f\right):=\sup\limits_{h\in I_n} \left\Vert f\left(\cdot-h\right)-f\left(\cdot\right)\right\Vert_p.
\end{equation*}

By applying this estimate, we immediately obtain that if $f\in lip\left(\alpha,p\right),$ i.e.,

\begin{equation*}
\omega_p\left(\frac{1}{M_n},f\right)=O\left(\frac{1} {M_n^{\alpha}}\right),\ \ n\rightarrow \infty,
\end{equation*}
then

\begin{equation*}
\left\Vert\sigma_nf-f\right\Vert_p=\left\{ \begin{array}{ll}
O\left(\frac{1}{M_N}\right),&\text{if}\  \alpha>1, \\
O\left(\frac{N}{M_N}\right),&\text{if} \ \alpha=1, \\
O\left(\frac{1}{M_N^{\alpha}}\right),& \text{if} \  \alpha<1. \end{array}\right.
\end{equation*} 

Another well-known summability method is the so called $\left(C,\alpha\right)$-means (Ces\`aro means) $\sigma_n^{\alpha}$, \index{\file}{Ces\`aro means} \index{\file-1}{$\left(C,\alpha\right)$} which are defined by \index{\file-1}{$\sigma_n^{\alpha}f$}

\begin{equation*}
\sigma_n^{\alpha}f:=\frac{1}{A_n^{\alpha}}\overset{n}{\underset{k=1}{\sum}}A_{n-k}^{\alpha-1}S_kf,
\end{equation*}
where \index{\file-1}{$A_n^{\alpha}$}
\begin{equation*}
A_0^{\alpha}:=0,\qquad A_n^{\alpha}:=\frac{\left(\alpha+1\right)...\left(\alpha+n\right)}{n!},\qquad \alpha \neq -1,-2,...
\end{equation*}
It is well-known that for $\alpha=1$ this summability method coincides with the Fej\'{e}r summation and for $\alpha=0$ we just have the partial sums of the Vilenkin-Fourier series.
Moreover, if we consider the maximal operator of the Ces\'aro means $\sigma^{\alpha ,\ast},$ defined by:
\index{\file}{maximal operator of Ces\'aro means}
\index{\file-1}{$\sigma^{\alpha,\ast}f$}
\begin{eqnarray*}
	\sigma^{\alpha ,\ast}f&:=&\sup_{n\in\mathbb{N}}\left\vert \sigma _n^{\alpha }f\right\vert
\end{eqnarray*}
for $0<\alpha\leq 1,$ then
\begin{eqnarray*}
	y \mu\left\{\sigma^{\alpha ,\ast}f>y \right\}\leq c\left\Vert f\right\Vert_1, \ \  f\in L^1(G_m), \ \ y>0.
\end{eqnarray*}
The boundedness of the maximal operator of the Reisz logarithmic means does not hold from $L^1(G_m)$ to the space $L^1(G_m).$ However,
\begin{equation*}
\left\Vert \sigma_n^{\alpha}f-f\right\Vert _{p}\rightarrow 0,\text{ \ \ \ when \ \ \ \ }n\rightarrow \infty ,\text{ \ }(f\in L^p(G_m),\text{ \ }1\leq p\leq \infty ).
\end{equation*}

Convergence and approximation in various norms of Vilenkin-Fejér means, $\left(C,\alpha\right)$-means and N\"orlund logarithmic means can be found in Blahota,  G{\'a}t and Goginava \cite{BGG2,BGG}, Blahota, and  Tephnadze \cite{bnt1,bpt1,bpt2,bt3,bt1,bt2,btt}, Fine \cite{fi,fi2}, Fridli \cite{FR}, Fujii \cite{Fu}, Goginava \cite{gog4,GoPubl}, Gogolashvili  Nagy and Tephnadze \cite{GNT,Gogolatep1,Gogolatep2}, Persson and Tephnadze \cite{pt1,pt2}  (see also \cite{BSPT,BNPT1,BPTW,LPTT,mpt,MST,PSTW,pttw,ptt1,ptw,ptw20,ptw2,ptw3}), P{\'a}l and Simon \cite{PS}, Schipp \cite{Sc,Sch2,s20}, Simon \cite{Si2,Si1}, Tephnadze \cite{tep91,tep10,tep1,tep4,tep20,tep9} (see also \cite{tep2,tep3,tep18,tep19,tep15,tep6,tep11,tep7,tep5,tep12,tep8,tep17,tep13,tepPhDGeo,tep14,tep16,tep21,tt1}, Tutberidze \cite{tut4,tut3}, Weisz \cite{We15,We32,we444} and Zhizhiashvili \cite{Zh3,Zh4,Zh1}. Similar problems for the two-dimensional case can be found in Nagy \cite{na,n,n1,nagy}, Nagy and Tephnadze \cite{NT3,NT1,NT2,NT5,NT4}.  \newpage

The properties established in Lemma \ref{lemma7kn} ensure that kernel of the Fejér means $\{K_N\}_{N=1}^\infty$ form what is called an approximation identity.

\begin{definition}
	The family $\{\Phi_n \}^{\infty}_{n=1}\subset L^{\infty}(G_m)$ forms an approximate identity provided that
	\begin{eqnarray*}
		&(A1)& \ \ \ \int_{G_m}\Phi_n(x)d(x)=1
		\label{1.71app}\\
		&(A2)& \ \ \ \sup_{n\in \mathbb{N}}\int_{G_m}\left\vert \Phi_n(x)\right\vert d\mu(x)<\infty\label{1.72app} \\
		&(A3)& \ \ \ \sup_{n\in\mathbb{N}}\int_{G_m \backslash I_N}\left\vert \Phi_n(x)\right\vert d\mu (x)\rightarrow  0, \ \ \text{as} \ \ n\rightarrow  \infty, \ \ \text{for any} \ \ N\in \mathbb{N_+}.
		\label{1.73app}
	\end{eqnarray*}
\end{definition}

The term "approximate identity" is used because of the fact that 

$$\Phi_n\ast f \to f \ \ \ \text{ as } \ \ \ n\to \infty$$ in any reasonable sense. In particular, the following results holds true (for details see the books \cite{Garsia} and \cite{MS1}):

\begin{proposition} \label{theoremconv}
	Let $f\in L^p(G_m),$ where $1\leq p\leq \infty$ and the family 
	
	$$\{\Phi_n \}^{\infty}_{n=1}\subset L^{\infty}(G_m)$$
	forms an approximate identity. Then
	
	$$\Vert \Phi_n\ast f - f \Vert_p \to 0 \ \ \text{ as } \ \ n\to \infty.$$
\end{proposition}

It is well-known that the $n$-th N\"orlund  mean $t_n$ and $T$ means $T_n$ \index{\file}{N\"orlund means} \index{\file}{$T$ means} for the Fourier series of $f$  are,  respectively, defined by \index{\file-1}{$t_nf$} \index{\file-1}{$T_nf$}
\begin{equation} \label{1.2}
t_nf:=\frac{1}{Q_n}\overset{n}{\underset{k=1}{\sum }}q_{n-k}S_kf
\end{equation}
and
\begin{equation} \label{1.2T}
T_nf:=\frac{1}{Q_n}\overset{n-1}{\underset{k=0}{\sum }}q_{k}S_kf,
\end{equation}
where $\{q_k:k\in \mathbb{N}\}$ \index{\file-1}{$q_k$} is a sequence of nonnegative numbers and

\index{\file-1}{$Q_n$}
\begin{equation*}
Q_n:=\sum_{k=0}^{n-1}q_k.
\end{equation*}

Let $\{q_k:k\geq 0\}$ be a sequence of nonnegative numbers where $q_0>0.$
	Then the summability method (\ref{1.2T}) generated by $\{q_k:k\geq 0\}$ is regular if and only if (see \cite{moo})
	
	\begin{equation*}\label{Tcond}
	\lim_{n\rightarrow\infty}\frac{q_{n-1}}{Q_n}=\infty.
	\end{equation*}

The representations
\begin{equation*}
t_nf\left(x\right)=\underset{G_m}{\int}f\left(t\right)F_n\left(x-t\right) d\mu\left(t\right) 
\end{equation*}
and

\begin{equation*}
T_nf\left(x\right)=\underset{G_m}{\int}f\left(t\right)F^{-1}_n\left(x-t\right) d\mu\left(t\right)
\end{equation*}
play central roles in the sequel, where \index{\file-1}{$F_n$}

\begin{equation}\label{1.3T}
F_n:=\frac{1}{Q_n}\overset{n}{\underset{k=1}{\sum }}q_{n-k}D_k
\end{equation}
and

\index{\file-1}{$F^{-1}_n$}
\begin{equation} \label{1.4T}
F^{-1}_n:=\frac{1}{Q_n}\overset{n}{\underset{k=1}{\sum }}q_{k}D_k
\end{equation}
are called the kernels of N\"orlund  and $T$ means, respectively.
\index{\file}{N\"orlund kernel}
\index{\file}{$T$ kernel}

N\"orlund are generalizations of  Fej\'er, $(C,\alpha)$ and  N\"orlund logarithmic means. According to all these facts it is of prior interest to study the behavior of operators related to N\"orlund means of
Fourier series with respect to orthonormal systems. The N\"orlund 
summation are general summability methods, which satisfy the conditions (A1)-(A3). This means that all N\"orlund and $T$ means are approximation identity. According to all these facts it is of prior interest to study the behavior of operators related to N\"orlund means of
Fourier series with respect to orthonormal systems. M\'oricz and Siddiqi \cite{Mor} investigated the approximation properties of
some special N\"orlund means of Walsh-Fourier series of $L^{p}$ functions in
norm. In particular, they proved that if $f\in L^p(G_m),$ $1\leq p\leq \infty,$ $n=M_j+k,$ $1\leq k\leq M_j \ (n\in \mathbb{N}_+)$ and $\{q_k:k\in \mathbb{N}_+\}$ is sequence of non-negative numbers, such that

$$
\frac{n^{\alpha-1}}{Q_n^{\alpha}}\sum_{k=0}^{n-1}q^{\alpha}_k =O(1),\ \ \text{for some} \ \ 1<\alpha\leq 2,
$$
then

$$
\Vert t_nf-f\Vert_p\leq \frac{C_p}{Q_n} \sum_{i=0}^{n-1}M_iq_{n-M_i}\omega_p\left(\frac{1}{M_i},f\right)+C_p\omega_p\left( \frac{1}{M_j},f\right),
$$
when $\{q_k:k\in\mathbb{N}\}$ is non-decreasing, while 

\begin{eqnarray*}
&&\Vert t_nf-f\Vert_p\\
&\leq& \frac{C_p}{Q_n} \sum_{i=0}^{n-1}\left(Q_{n-M_j+1}-Q_{n-M_{j+1}+1} \right)\omega_p\left(\frac{1}{M_i},f\right) +C_p\omega_p\left( \frac{1}{M_j},f\right),
\end{eqnarray*}
when  $\{q_k:k\in \mathbb{N}\}$ is non-increasing.

Let us define maximal operator of N\"orlund means by

\begin{eqnarray*}
t^{\ast}f:=\sup_{n\in\mathbb{N}}\left\vert t_nf\right\vert.
\end{eqnarray*}

If $\{q_k:k\in\mathbb{N}\}$ is non-increasing and satisfying the condition
\begin{equation} \label{fn0}
\frac{1}{Q_n}=O\left( \frac{1}{n}\right) ,\text{ \ \ as \ \ }
n\rightarrow\infty,
\end{equation}
or if $\{q_k:k\in\mathbb{N}\}$ is non-decreasing, satisfying the condition 

\begin{equation}\label{fn01}
\frac{q_{n-1}}{Q_n}=O\left( \frac{1}{n}\right) ,\text{ \ \ as \ \ }
n\rightarrow\infty,
\end{equation}
then

\begin{eqnarray*}
y \mu\left\{t^{\ast}f>y\right\}\leq c\left\Vert f\right\Vert_{1}, \ \ \ f\in L^1(G_m), \ \ y>0.
\end{eqnarray*}

The boundedness of the such maximal operator of N\"orlund means does not hold from $L^1(G_m)$ to the space $L^1(G_m).$ However

\begin{equation*}
\left\Vert t_nf-f\right\Vert_p\rightarrow 0,\text{ \ \ \ as \ \ \ \ }n\rightarrow \infty ,\text{ \ }(f\in L^p(G_m),\text{ \ }1\leq p\leq \infty ).
\end{equation*}

\section{Well-known and New examples of N\"orlund Means }\label{s3.2}
\markright{Well-known ans new examples of N\"orlund Means}

\ \ \ \ We define  $B_n$ means as the class of N\"orlund means, with monotone and bounded sequence $\{q_k:k\in \mathbb{N}\}$, such that

\begin{equation*}
0<q<\infty \ \ \ \text{where} \ \ \ q_{\infty}:=\lim_{n\rightarrow\infty}q_n.
\end{equation*}

If the sequence $\{q_k:k\in\mathbb{N}\}$ is non-decreasing, then we have
that

\begin{equation*}
nq_0\leq Q_n\leq nq_\infty.
\end{equation*}

In the case when the sequence $\{q_{k}:k\in \mathbb{N}\}$ is non-increasing,
then

\begin{equation*}
nq_\infty\leq Q_{n}\leq nq_{0}.  \label{monotone0}
\end{equation*}

In both cases we can conclude that conditions \eqref{fn01} and \eqref{fn0} are fulfilled.

%

Well-known examples of  N\"orlund  means with monotone and bounded sequence $\{q_k:k\in \mathbb{N}\}$ is Fejér means 

\begin{equation*}
\sigma_nf:=\frac{1}{n}\sum_{k=1}^nS_kf
\end{equation*}
It is evident that in this case conditions \eqref{fn01} and \eqref{fn0} are fulfilled.

The Ces\`aro means $\sigma_n^{\alpha}$  (sometimes also denoted $\left(C,\alpha\right)$)  \index{\file-1}{$\left(C,\alpha\right)$} are defined by
\begin{equation*}
\sigma_n^{\alpha}f:=\frac{1}{A_n^{\alpha}}\overset{n}{\underset{k=1}{\sum}}A_{n-k}^{\alpha-1}S_kf
\end{equation*}
where 
\begin{equation*}
A_0^{\alpha}:=0,\qquad A_n^{\alpha}:=\frac{\left(\alpha+1\right)...\left(\alpha+n\right)}{n!},\qquad \alpha \neq -1,-2,...
\end{equation*}

It is well-known that  \index{\file-1}{$A_n^{\alpha}$}
\begin{equation} \label{node0}
A_n^{\alpha}=\overset{n}{\underset{k=0}{\sum}}A_{n-k}^{\alpha-1},
\end{equation}%
\begin{equation} \label{node01}
A_n^{\alpha}-A_{n-1}^{\alpha}=A_n^{\alpha-1}\ \ \ \text{and} \ \ \
A_{n}^{\alpha }\backsim n^{\alpha }.
\end{equation}

It is obvious that
\begin{equation} \label{node1C}
\frac{\left\vert q_n-q_{n+1}\right\vert}{n^{\alpha-2}}=O\left(1\right)
,\qquad \text{as}\qquad n\rightarrow\infty,
\end{equation}
\begin{equation} \label{node2C}
\frac{q_0}{Q_n}=O\left(\frac{1}{n^{\alpha }}\right),\qquad \ \ \ \ \ \text{as} \qquad n\rightarrow\infty,
\end{equation}
and
\begin{equation} \label{node3C}
\frac{q_{n-1}}{Q_n}=O\left(\frac{1}{n}\right),\qquad \ \ \ \ \ \text{as} \qquad  n\rightarrow\infty.
\end{equation}

Let $V_n^{\alpha}$ denote the N\"orlund mean, where

\begin{equation*}
\left\{q_k=\left(k+1\right)^{\alpha-1}:\quad k\in \mathbb{N}, \quad 0<\alpha<1\right\} , 
\end{equation*}
that is \index{\file-1}{$V_n^{\alpha}f$}

\begin{equation*}
V_n^{\alpha}f:=\frac{1}{Q_n}\overset{n}{\underset{k=1}{\sum }}\left(n-k-1\right)^{\alpha-1}S_kf.
\end{equation*}

It is obvious that

\begin{equation} \label{node1}
\frac{\left\vert q_n-q_{n+1}\right\vert}{n^{\alpha-2}}=O\left(1\right)
,\qquad \text{as}\qquad n\rightarrow\infty,
\end{equation}
\begin{equation} \label{node2}
\frac{q_0}{Q_n}=O\left(\frac{1}{n^{\alpha }}\right),\qquad \ \ \ \ \ \text{as} \qquad n\rightarrow\infty,
\end{equation}
and
\begin{equation} \label{node3}
\frac{q_{n-1}}{Q_n}=O\left(\frac{1}{n}\right),\qquad \ \ \ \ \ \text{as} \qquad  n\rightarrow\infty.
\end{equation}
We just remind again that $n$-th N\"orlund $L_n$ and Riesz $R_n$ logarithmic means are
defined by the sequence 
$
\left\{q_k=1/k,\quad k\in \mathbb{N}_+\right\}: 
$

\begin{eqnarray*}
	L_nf:=\frac{1}{l_n}\sum_{k=0}^{n-1}\frac{S_kf}{n-k},\ \ \
	R_nf:=\frac{1}{l_n}\sum_{k=1}^{n}\frac{S_kf}{k},
\end{eqnarray*}
respectively, where 

$$
l_n:=\sum_{k=1}^{n}\frac{1}{k}.
$$

It is evident that

\begin{equation} \label{NRL1}
\frac{q_{n-1}}{Q_n}=O\left(\frac{1}{n}\right),\text{ \  \ as \ \ }
n\rightarrow \infty
\end{equation}
and

\begin{equation} \label{NRL2}
\frac{q_0}{Q_n}=O\left(\frac{1}{\ln n}\right),\text{ \  as \ }
n\rightarrow \infty.
\end{equation}

Let $U^\alpha_n$ denote the N\"orlund mean, where

\begin{equation*}
\left\{ q_k=\frac{1}{(k+3){\ln^\alpha (k+3)}} \ : \ k\in\mathbb{N}, \ 0<\alpha\leq 1\right\} ,
\end{equation*}
that is \index{\file-1}{$U_n^\alpha f$}

\begin{equation*}
U_n^\alpha f:=\frac{1}{Q_n}\overset{n}{\underset{k=1}{\sum }}\frac{S_kf}{\left(n-k-3\right)\ln^\alpha\left(n-k-3\right)}.
\end{equation*}

It is obvious that
\begin{equation} \label{node4}
\frac{q_{n-1}}{Q_n}=O\left(\frac{1}{n}\right),\qquad \text{as} \qquad n\rightarrow\infty
\end{equation}
and

\begin{equation} \label{node5}
\frac{q_0}{Q_n}=\left\{
\begin{array}{c}
O\left(1/\ln(\ln n\right)),\text{ \ \ \ if \ \ \ }\alpha=1,\\
O\left( 1/\ln^{1-\alpha }n\right) ,\text{ \ \ \ \ \  if }0<\alpha <1.\end{array}\right.
\end{equation}

Let $\alpha\in\mathbb{R}_+$.
If we define the sequence $\{q_k:k\in \mathbb{N}\}$ by

\begin{equation*}
\left\{q_k=\log^{\alpha}(k+1) \ : \ k\in\mathbb{N} \ : \ \alpha>0\right\},
\end{equation*}
then we get the class of N\"orlund means with non-decreasing coefficients: \index{\file-1}{$\beta_n^{\alpha}f$}

\begin{equation*}
\beta_n^{\alpha}f:=\frac{1}{Q_n}\sum_{k=1}^{n}\log^{\alpha}\left( n-k-1\right)S_kf.
\end{equation*}%

%

It is obvious that

\begin{equation*}
\frac{n}{2}\log^{\alpha}\left(n/2\right)\leq Q_n\leq n\log^{\alpha}n.
\end{equation*}

It follows that

\begin{eqnarray} \label{node00}
\frac{1}{Q_n}\leq\frac{c}{n\log^{\alpha }n}= O\left(\frac{1}{n}\right)\rightarrow 0,\text{ \ as \ }n\rightarrow \infty.
\end{eqnarray}
and

\begin{eqnarray} \label{node001}
\frac{q_{n-1}}{Q_n}\leq\frac{c\log^{\alpha}\left(n-1\right)}{n\log^\alpha n}= O\left(\frac{1}{n}\right)\rightarrow 0,\text{ \ as \ }n\rightarrow \infty.
\end{eqnarray}


\section{Kernels of N\"orlund Means}\label{s3,4a}
\markright{Kernels of N\"orlund Means}

\ \ \ \ Now we study kernels of  N\"orlund means with respect to Vilenkin systems. 
If we  invoke Abel transformations for $a_j=A_j-A_{j-1}, \ j=1,...,n,$
\begin{eqnarray}\label{abel1N} \overset{n}{\underset{j=1}{\sum}}a_jb_{n-j}&=&A_{n}b_{0}+\overset{n-1}{\underset{j=1}{\sum}}A_j(b_j-b_{j+1}),  \\ \label{abel2N}
\overset{n}{\underset{j=M_N}{\sum}}a_jb_{n-j}&=&A_{n}b_{0}-A_{M_N-1}b_{n-M_N}+\overset{n-1}{\underset{j=M_N}{\sum}}A_j(b_j-b_{j+1}),
\end{eqnarray} 
when $b_j=q_j$, $a_j=1$ and $A_j=j$ for $j=0,1,...,n,$ then \eqref{abel1N} and \eqref{abel2N}  give the following identities:
\begin{eqnarray} \label{2b}
Q_n&:=&\overset{n-1}{\underset{j=0}{\sum}}q_j=\overset{n}{\underset{j=1}{\sum }}q_{n-j}\cdot 1 =\overset{n-1}{\underset{j=1}{\sum}}\left(q_{n-j}-q_{n-j-1}\right) j+q_0n,\\
\label{2b1}
&=&\overset{n-1}{\underset{j=M_N}{\sum}}q_{n-j}=\overset{n-1}{\underset{j=M_N}{\sum }}q_{n-j}\cdot 1 \\ \notag
&=&\overset{n-1}{\underset{j=M_N}{\sum}}\left(q_{n-j}-q_{n-j-1}\right) j+q_0n-(M_N-1) q_{n-M_N}.
\end{eqnarray}
Moreover, if  we instead use the Abel transformations \eqref{abel1N} and \eqref{abel2N} for $b_j=q_{n-j}$, $a_j=D_j$ and $A_j=jK_j$ for any $j=0,1,...,n-1$ we get the identities: 
\begin{eqnarray} \label{2bb}
F_n&=&\frac{1}{Q_n}\left(\overset{n-1}{\underset{j=1}{\sum}}\left(
q_{n-j}-q_{n-j-1}\right) jK_{j}+q_0nK_n\right),\\ \label{2bb1}
&=&\frac{1}{Q_n}\overset{n}{\underset{j=M_N}{\sum}}q_{n-j}D_j\\ \notag
&=&\frac{1}{Q_n}\left(\overset{n-1}{\underset{j=M_N}{\sum}}\left(
q_{n-j}-q_{n-j-1}\right) jK_{j}+q_0nK_n-q_{n-M_N}(M_N-1)K_{M_N-1} \right).
\end{eqnarray}
Analogously, if  we use the Abel transformations \eqref{abel1N} and \eqref{abel2N} for $b_j=q_j$, $a_j=S_j$ and $A_j=j\sigma_j$ for any $j=0,1,...,n-1$ we get the identities: 
\begin{eqnarray} \label{2bbb}
t_nf&=&\frac{1}{Q_n}\left(\overset{n-1}{\underset{j=1}{\sum}}\left(
q_{n-j}-q_{n-j-1}\right) j\sigma_{j}f+q_0n\sigma_nf\right)\\ \label{2bbb1}
&&\frac{1}{Q_n}\overset{n}{\underset{j=M_N}{\sum}}q_{n-j}S_jf\\ \notag
&=&\frac{1}{Q_n}\left(\overset{n-1}{\underset{j=M_N}{\sum}}\left(
q_{n-j}-q_{n-j-1}\right) j\sigma_{j}f+q_0nK_n-q_{n-M_N}(M_N-1)\sigma_{M_N-1}f \right).
\end{eqnarray}

First we consider N\"orlund kernels with respect to Vilenkin systems, which are generated by non-decreasing sequences:

\begin{lemma}\label{lemma0nn}
	Let $\{q_k:k\in\mathbb{N}\}$ be a sequence of non-decreasing numbers, satisfying the condition \eqref{fn01}.
	Then
	
	\begin{equation*}
	\left\vert F_n\right\vert\leq\frac{c}{n}\left\{\sum_{j=0}^{\left\vert n\right\vert }M_j\left\vert K_{M_j}\right\vert \right\},
	\end{equation*}
	where $c$ is an absolute constant.
\end{lemma}

\begin{proof}
	Let the sequence $\{q_k:k\in \mathbb{N}\}$ be non-decreasing. Then, by using (\ref{fn01}), we get that
	\begin{eqnarray*}
		\frac{1}{Q_n}\left(\overset{n-1}{\underset{j=1}{\sum }}\left\vert
		q_{n-j}-q_{n-j-1}\right\vert+q_0\right) 
		&\leq&\frac{1}{Q_n}\left(\overset{n-1}{\underset{j=1}{\sum }}\left(
		q_{n-j}-q_{n-j-1}\right)+q_0\right) \\
		&\leq &\frac{q_{n-1}}{Q_{n}}\leq \frac{c}{n}.
	\end{eqnarray*}
	
	Hence, in view of (\ref{fn01}) if we apply (\ref{fn5}) in Corollary \ref%
	{lemma7kn0} and use the equality (\ref{2bb}) we obtain that
	\begin{eqnarray*}
		\left\vert F_n\right\vert &\leq & \left( \frac{1}{Q_n}\left( \overset{n-1}{\underset{j=1}{\sum }}\left\vert q_{n-j}-q_{n-j-1} \right\vert+q_0\right)\right)\sum_{i=0}^{\left\vert n\right\vert } M_i\left\vert K_{M_i}\right\vert \\
		&=&\left(\frac{1}{Q_n}\left(\overset{n-1}{\underset{j=1}{\sum}}\left(
		q_{n-j}-q_{n-j-1}\right)+q_0\right)\right) \sum_{i=0}^{\left\vert
			n\right\vert}M_i\left\vert K_{M_i}\right\vert \\
		&\leq & \frac{q_{n-1}}{Q_n}\sum_{i=0}^{\left\vert n\right\vert
		}M_i\left\vert K_{M_i}\right\vert \leq\frac{c}{n}\sum_{i=0}^{\left\vert n\right\vert }M_i\left\vert K_{M_i}\right\vert.
	\end{eqnarray*}
	The proof is complete.
\end{proof}

We also state analogical estimate, but now without any restriction like (\ref{fn01}):

\begin{lemma}\label{lemma00nn} 
	Let $n\geq M_N$ and $\{q_k:k\in\mathbb{N}\}$ be a
	sequence of non-decreasing numbers. Then
	
	\begin{equation*}
	\left\vert \frac{1}{Q_n}\overset{n}{\underset{j=M_N}{\sum}}
	q_{n-j}D_j\right\vert\leq\frac{c}{M_N}\left\{\sum_{j=0}^{\left\vert
		n\right\vert }M_j\left\vert K_{M_j}\right\vert\right\},
	\end{equation*}
	where $c$ is an absolute constant.
\end{lemma}

\begin{proof}
	Let $M_N-1 \leq j\leq n.$ In the view of (\ref{fn5}) in Corollary \ref{lemma7kn0} we find that
	
	\begin{eqnarray*}
		\left\vert K_j\right\vert &\leq& \frac{1}{j}\sum_{l=0}^{\left\vert
			j\right\vert }M_l\left\vert K_{M_l}\right\vert  \leq \frac{1}{M_N-1}\sum_{l=0}^{\left\vert n\right\vert }M_l\left\vert K_{M_l}\right\vert\\
		&\leq & \frac{c}{M_N}\sum_{l=0}^{\left\vert n\right\vert }M_l\left\vert K_{M_l}\right\vert
	\end{eqnarray*}
	
	Since the sequence $\{q_k:k\in \mathbb{N}\}$ be non-decreasing we get that
	
	\begin{eqnarray*}
		q_{n-M_N}(M_N-1)&=&\overset{M_N-1}{\underset{j=0}{\sum}}q_{n-M_N}\leq \overset{M_N-1}{\underset{j=0}{\sum}}q_{n-M_N+j}\\
		&=&q_{n-M_N}+q_{n-M_N+2}+...+q_{n-M_N+(M_N-1)}\leq Q_n
	\end{eqnarray*}
	and
	
	\begin{eqnarray*}
		&&\overset{n-1}{\underset{j=M_N}{\sum }}\left\vert
		q_{n-j}-q_{n-j-1}\right\vert j+q_0n+q_{n-M_N}(M_N-1) \\
		&\leq& \overset{n-1}{\underset{j=1}{\sum}}\left\vert q_{n-j}-q_{n-j-1}\right\vert j+q_0n+q_{n-M_N}(M_N-1) \\
		&=&\overset{n-1}{\underset{j=1}{\sum}}\left(q_{n-j}-q_{n-j-1}\right)
		j+q_0n+q_{n-M_N}(M_N-1)\\
		&=&Q_n+q_{n-M_N}(M_N-1)\leq 2 Q_n.
	\end{eqnarray*}
	
	By using the Abel transformation \eqref{2bb1} we can conclude that
	\begin{eqnarray*}
		&&\left\vert \frac{1}{Q_n}\overset{n}{\underset{j=M_{N}}{\sum }}
		q_{n-j}D_j\right\vert \\
		&=&\left\vert\frac{1}{Q_n}\left( \overset{n-1}{\underset{j=M_N}{\sum}}
		\left(q_{n-j}-q_{n-j-1}\right)jK_j+q_0nK_n+q_{n-M_N}(M_N-1)K_{M_N-1}\right)\right\vert \\
		&\leq& \left(\frac{1}{Q_n}\left(\overset{n-1}{\underset{j=M_N}{\sum}}
		\left\vert q_{n-j}-q_{n-j-1}\right\vert j+q_0n\right)\right)\frac{c}{
			M_N}\sum_{i=0}^{\left\vert n\right\vert }M_i\left\vert K_{M_i}\right\vert \\
		&\leq&\frac{c}{M_N}\sum_{i=0}^{\left\vert n\right\vert}M_i\left\vert
		K_{M_i}\right\vert.
	\end{eqnarray*}
	The proof is complete.
\end{proof}

Now we prove a lemma, which is very important for our further investigation to prove norm convergence in Lebesgue spaces of N\"orlund means generated by non-decreasing sequences $\{q_k:k\in \mathbb{N}\}$ in this chapter.

\begin{lemma}\label{Corollary3nn101} Let $\{q_k:k\in \mathbb{N}\}$ be a sequence of non-decreasing numbers. Then, for any $n, N\in \mathbb{N_+}$,
	\begin{eqnarray} \label{1.71}
	&&\int_{G_m} F_n(x) d\mu (x)=1, \\
	&&\sup_{n\in\mathbb{N}}\int_{G_m}\left\vert F_n(x)\right\vert d\mu(x)\leq c<\infty,\label{1.72} \\
	&&\sup_{n\in\mathbb{N}}\int_{G_m \backslash I_N}\left\vert F_n(x)\right\vert d\mu (x)\rightarrow  0, \ \ \text{as} \ \ n\rightarrow  \infty, \label{1.73}
	\end{eqnarray}
	where $c$ is an absolute constant.
\end{lemma}

\begin{proof}
	According to Lemma \ref{dn2.6int1}  we readily obtain \eqref{1.71}. 
	By using \eqref{fn4} in Corollary \ref{lemma7kn} combined with \eqref{2b} and \eqref{2bb} we get that
	\begin{eqnarray*}
		\int_{G_m }\left\vert F_n\right\vert d\mu
		&\leq&\frac{1}{Q_n}\overset{n-1}{\underset{j=1}{\sum}}\left(q_{n-j}-q_{n-j-1}\right)j\int_{G_m}\left\vert K_j\right\vert d\mu 
		+\frac{q_0n}{Q_n}\int_{G_m}\vert K_n\vert d\mu\\
		&\leq&\frac{c}{Q_n}\overset{n-1}{\underset{j=1}{\sum}}\left(q_{n-j}-q_{n-j-1}\right)j+\frac{cq_0n}{Q_n}<c<\infty,
	\end{eqnarray*}
	so also \eqref{1.72} is proved.
	
	According to \eqref{fn400} in Corollary \ref{lemma7kn} and also \eqref{2b} and \eqref{2bb} we find that
	\begin{eqnarray*}
		\int_{G_m \backslash I_N}\left\vert F_n\right\vert d\mu 
		&\leq&\frac{1}{Q_n}\overset{n-1}{\underset{j=0}{\sum}}\left(q_{n-j}-q_{n-j-1}\right)j\int_{G_m \backslash I_N}\left\vert K_j\right\vert d\mu \\ &+&\frac{q_0n}{Q_n}\int_{G_m \backslash I_N}\vert K_n\vert\\
		&\leq&\frac{1}{Q_n}\overset{n-1}{\underset{j=0}{\sum}}\left(q_{n-j}-q_{n-j-1}\right)j\alpha_j+\frac{q_0n\alpha_n}{Q_n},\\ \notag
		&:=&I+II.
	\end{eqnarray*}
	where  $\alpha_n\to 0, \ \ \text{as} \ \ n\to\infty.$
		Since the sequence is non-decreasing, we can conclude that	
	$$II=\frac{q_{0}n\alpha_n}{Q_n}\leq \alpha_n\to 0, \ \ \text{as} \ \ n\to\infty.$$
	On the other hand, for any $\varepsilon>0$ there exists $N_0\in \mathbb{N},$ such that
	$$\alpha_n< \varepsilon, \ \ \text{ when }  \ \ n>N_0.$$ 
	Moreover,	
	\begin{eqnarray*}
		I&=&\frac{1}{Q_n}\overset{n-1}{\underset{j=1}{\sum}}\left(q_{n-j}-q_{n-j-1}\right)j\alpha_j \\
		&=&\frac{1}{Q_n}\overset{N_0}{\underset{j=1}{\sum}}\left(q_{n-j}-q_{n-j-1}\right)j\alpha_j
		+\frac{1}{Q_n}\overset{n-1}{\underset{j=N_0+1}{\sum}}\left(q_{n-j}-q_{n-j-1}\right)j\alpha_j
		=I_1+I_2.
	\end{eqnarray*}
	Since sequence is non-decreasing, we can conclude that 
	$$\vert q_{n-j}-q_{n-j-1}\vert<2q_{n-1}$$
	\begin{eqnarray*}
		I_1=\frac{1}{Q_n}\overset{N_0}{\underset{j=0}{\sum}}\left(q_{n-j}-q_{n-j-1}\right)j\alpha_j \leq \frac{2q_{n-1}N_0}{Q_n}\to 0, \ \ \ \text{as} \ \ \ n\to \infty
	\end{eqnarray*}
	and, furthermore,
	\begin{eqnarray*}
		I_2=\frac{1}{Q_n}\overset{n-1}{\underset{j=N_0+1}{\sum}}\left(q_{n-j}-q_{n-j-1}\right)j\alpha_j
		&\leq& \frac{\varepsilon}{Q_n}\overset{n-1}{\underset{j=N_0+1}{\sum}}\left(q_{n-j}-q_{n-j-1}\right)j\\
		&\leq& \frac{\varepsilon}{Q_n}\overset{n-1}{\underset{j=0}{\sum}}\left(q_{n-j}-q_{n-j-1}\right)j<\varepsilon,
	\end{eqnarray*}
	and it follows that also $I\to 0$ so also \eqref{1.73} is proved.
	The proof is complete.
\end{proof}


We also consider the kernel of N\"orlund means with respect to the Vilenkin systems which are generated by non-increasing sequences $\{q_k:k\in\mathbb{N}\}$, but now with some new restrictions on the indexes:

\begin{lemma} \label{lemma0nn1}
	Let $\{q_k:k\in\mathbb{N}\}$ be a sequence of
	non-increasing numbers satisfying the condition \eqref{fn0}.
	Then
	
	\begin{equation*}
	\left\vert F_n\right\vert \leq\frac{c}{n}\left\{\sum_{j=0}^{\left\vert n\right\vert }M_j\left\vert K_{M_j}\right\vert\right\},
	\end{equation*}
	where $c$ is an absolute constant. \\
\end{lemma}

\begin{proof}
	Let the sequence $\{q_k:k\in \mathbb{N}\}$ be non-increasing and satisfying
	condition (\ref{fn0}). Then
	\begin{eqnarray*}
		\frac{1}{Q_n}\left( \overset{n-1}{\underset{j=1}{\sum}}\left\vert
		q_{n-j}-q_{n-j-1}\right\vert+q_0\right) 
		&\leq& \frac{1}{Q_n}\left(\overset{n-1}{\underset{j=1}{\sum}}-
		\left(q_{n-j}-q_{n-j-1}\right)+q_0\right) \\
		&\leq&\frac{2q_0-q_{n-1}}{Q_n}\leq\frac{2q_0}{Q_n}\leq\frac{c}{n}.
	\end{eqnarray*}
	
	Hence, if we apply (\ref{fn5}) in Corollary \ref{lemma7kn0} and invoke equalities (\ref{2b}) and (\ref{2bb}), then we get that
	\begin{eqnarray*}
		\left\vert F_n\right\vert &\leq&\left(\frac{1}{Q_n} \left(\overset{n-1}{\underset{j=1}{\sum }}\left\vert q_{n-j}-q_{n-j-1}\right\vert +q_0\right)\right) \sum_{i=0}^{\left\vert n\right\vert }M_{i}\left\vert K_{M_i}\right\vert \\
		&=&\left(\frac{1}{Q_n}\left(\overset{n-1}{\underset{j=1}{\sum}}-\left(
		q_{n-j}-q_{n-j-1}\right)+q_0\right)\right) \sum_{i=0}^{\left\vert
			n\right\vert }M_{i}\left\vert K_{M_{i}}\right\vert \\
		&\leq& \frac{2q_0-q_{n-1}}{Q_n}\sum_{i=0}^{\left\vert n\right\vert
		}M_i\left\vert K_{M_i}\right\vert  \\
	&\leq&  \frac{2q_0}{Q_n}\sum_{i=0}^{\left\vert n\right\vert}M_i\left\vert K_{M_i}\right\vert \\
	&\leq& \frac{c}{n}\sum_{i=0}^{\left\vert n\right\vert }M_i\left\vert
		K_{M_i}\right\vert.
	\end{eqnarray*}
	
	The proof is complete.
\end{proof}

The next  result is very important for our further investigation in this Chapter to prove norm convergence in Lebesgue spaces of N\"orlund means generated by a non-increasing sequence $\{q_k:k\in \mathbb{N}\}$:

\begin{corollary}
	\label{corollary3n9} Let $\{q_k:k\in \mathbb{N}\}$ be a sequence of non-increasing numbers satisfying the condition \eqref{fn0}. Then, for any $n,N\in \mathbb{N_+}$,
	
	\begin{eqnarray} \label{1.71inc}
	&&\int_{G_m} F_n(x) d\mu (x)=1, \\
	&&\sup_{n\in\mathbb{N}}\int_{G_m}\left\vert F_n(x)\right\vert d\mu(x)\leq c<\infty,\label{1.72inc} \\
	&&\sup_{n\in\mathbb{N}}\int_{G_m \backslash I_N}\left\vert F_n(x)\right\vert d\mu (x)\rightarrow  0, \ \ \text{as} \ \ n\rightarrow  \infty, \label{1.73inc}
	\end{eqnarray}
	where $c$ is an absolute constant.
\end{corollary}

\begin{proof}
	If we compare the estimation of $K_n$ in Lemma \ref{lemma7kn0} or  $F_n$ in Lemma \ref{lemma0nn} with the estimation of $F_n$ in Lemma \ref{lemma0nn1} we find that they are quite the same. Hence, the proof is analogous to those of Corollary \ref{lemma7kn} and Lemma \ref{Corollary3nn101}, so, we leave out the details.
	
\end{proof}


Finally we study some special subsequences of kernels of N\"orlund and $T$ means:
\begin{lemma}\label{lemma0nnT121}Let $n\in \mathbb{N}$. Then	
	
	\begin{eqnarray} \label{1.71alphaT2j} F_{M_n}(x)=D_{M_n}(x)-\psi_{M_n-1}(x)\overline{F^{-1}}_{M_n}(x)
	\end{eqnarray}
	and
	
	\begin{eqnarray} \label{1.71alphaT2j1} F^{-1}_{M_n}(x)=D_{M_n}(x)-\psi_{M_n-1}(x)\overline{F}_{M_n}(x).
	\end{eqnarray}
\end{lemma}
\begin{proof} By using \eqref{dn22} in Lemma \ref{dn} we get that
	
	\begin{eqnarray*}
		F_{M_n}(x)&=&\frac{1}{Q_{M_n}}\overset{M_n}{\underset{k=1}{\sum}}q_{M_n-k}D_k(x)=\frac{1}{Q_{M_n}}\overset{M_n-1}{\underset{k=0}{\sum }}q_kD_{M_n-k}(x)\\
		&=&\frac{1}{Q_{M_n}}\overset{M_n-1}{\underset{k=0}{\sum }}q_k\left(D_{M_n}(x)-\psi_{M_n-1}(x)\overline{D}_j(x)\right)\\
		&=&D_{M_n}(x)-\psi_{M_n-1}(x)\overline{F^{-1}}_{M_n}(x)
	\end{eqnarray*}
	Hence, \eqref{1.71alphaT2j} is proved. Identity \eqref{1.71alphaT2j1} is proved analogously so
	the proof is complete.
	
\end{proof}

Next four lemmas will be used to prove norm convergence and almost everywhere convergence of subsequences of N\"orlund  means:

\begin{corollary} \label{corollary3nn}
	Let $\{q_k:k\in \mathbb{N}\}$ be a sequence of non-decreasing numbers. Then, for any  $n,N\in \mathbb{N_+}$,
	
	\begin{eqnarray} \label{1.71alpha-}
	&&\int_{G_m} F^{-1}_{M_n}(x) d\mu (x)=1, \\
	&&\sup_{n\in\mathbb{N}}\int_{G_m}\left\vert F^{-1}_{M_n}(x)\right\vert d\mu(x)\leq c<\infty,\label{1.72alpha-} \\
	&&\sup_{n\in\mathbb{N}}\int_{G_m \backslash I_N}\left\vert F^{-1}_{M_n}(x)\right\vert d\mu (x)\rightarrow  0, \ \ \text{as} \ \ n\rightarrow  \infty, \label{1.73alpha-}
	\end{eqnarray}
\end{corollary}
\begin{proof}
	According to  \eqref{1.71alphaT2j1} the proof is a direct consequence of Lemmas \ref{dn2.3} and \ref{Corollary3nn101} and Lemma \ref{dn2.6int1}. The proof is complete.
	
\end{proof}

\begin{corollary} \label{corollary3n}
	Let $\{q_k:k\in \mathbb{N}\}$ be a sequence of non-increasing numbers. Then, for any $ N\in \mathbb{N_+},$
	
	\begin{eqnarray} \label{1.71alpha}
	&&\int_{G_m} F_{M_n}(x) d\mu (x)=1, \\
	&&\sup_{n\in\mathbb{N}}\int_{G_m}\left\vert F_{M_n}(x)\right\vert d\mu(x)\leq c<\infty,\label{1.72alpha} \\
	&&\sup_{n\in\mathbb{N}}\int_{G_m \backslash I_N}\left\vert F_{M_n}(x)\right\vert d\mu (x)\rightarrow  0, \ \ \text{as} \ \ n\rightarrow  \infty, \label{1.73alpha}
	\end{eqnarray}
\end{corollary}
\begin{proof} According to \eqref{1.71alphaT2j} the proof is a direct consequence of Lemmas \ref{dn2.3}, \ref{dn2.6int1} and \ref{lemma0nnT12}. The proof is complete.
	
\end{proof}
\newpage

\section{Norm Convergence of N\"orlund Means in Lebesgue Spaces}\label{s3,5}
\markright{Norm Convergence of N\"orlund  Means}

\ \ \ \ First we consider norm convergence of N\"orlund means with respect to Vilenkin systems:

\begin{theorem}\label{Corollary3nnconv1} Let  $f\in L^p(G_m)$ for $p\geq 1$ and $\{q_k:k\in \mathbb{N}\}$ be a sequence of non-decreasing numbers. Then
	
	$$\Vert t_n f-f\Vert_p \to 0 \ \ \text{as}\ \ n\to \infty.$$
\end{theorem}

\begin{proof}
	According to Lemma \ref{Corollary3nn101} we conclude that the conditions (A1), (A2) and (A3) in Theorem \ref{theoremconv} are fulfilled, which implies the stated norm convergence.
	
	The proof is complete.
\end{proof}

\begin{theorem}\label{Corollary3nnconv111} Let  $f\in L^p(G_m)$ for $p\geq 1$ and $\{q_k:k\in \mathbb{N}\}$ be a sequence of non-increasing numbers satisfying the condition \eqref{fn0}. Then
	
	$$\Vert t_n f-f\Vert_p \to 0 \ \ \text{as}\ \ n\to \infty.$$
\end{theorem}

\begin{proof}
	According to Corollary \ref{corollary3n9} we conclude that the conditions (A1), (A2) and (A3) in Theorem \ref{theoremconv} are fulfilled and the stated norm convergence follows.
	
	The proof is complete.
\end{proof}

According to Theorems \ref{Corollary3nnconv1} and \ref{Corollary3nnconv111} we get the following result for N\"orlund  means:
\begin{corollary}\label{Corollary3nnconv1111}
	Let  $f\in L^p(G_m)$ for $p\geq 1$. Then	
	
	\begin{eqnarray*}
		&&\Vert \sigma_{n}f-f\Vert_p \to 0 \ \ \text{as}\ \ n\to \infty,\\
		&&\Vert B_{n}f-f\Vert_p \to 0 \ \ \text{as}\ \ n\to \infty,\\
		&&\Vert \beta^{\alpha}_{n}f-f\Vert_p \to 0 \ \ \text{as}\ \ n\to \infty,\\
	\end{eqnarray*}
\end{corollary}

\begin{proof}
	Since $\sigma_{n}f$ and $\beta^{\alpha}_{n}f$ are N\"orlund  means generated by non-decreasing sequences $\{q_k:k\in \mathbb{N}\},$ the corresponding norm convergences are direct consequences of Theorem \ref{Corollary3nnconv1}. In the case of  $B_{n}f$ means with non-decreasing sequence $\{q_k:k\in \mathbb{N}\},$ this result is also a consequence of Theorem \ref{Corollary3nnconv1}. 
	
	On the other hand, in the case of  $B_{n}f$ means with non-increasing sequence $\{q_k:k\in \mathbb{N}\},$ this result is a consequence of Theorem \ref{Corollary3nnconv111} and \eqref{fn0}. 
	
	The proof is complete.
\end{proof}
Now, we consider  subsequences of N\"orlund  means, generated by non-increasing sequences, but without any restrictions on the sequence $\{q_k:k\in \mathbb{N}\}$:
\begin{theorem}\label{Corollaryconv41} Let  $f\in L^p(G_m)$ for $p\geq 1$ and $\{q_k:k\in \mathbb{N}\}$ be a sequence of non-increasing numbers. Then	
	$$\Vert t_{M_n} f-f\Vert_p \to 0 \ \ \text{as}\ \ n\to \infty.$$
\end{theorem}

\begin{proof}
	According to Corollary \ref{corollary3n} we conclude that conditions (A1), (A2) and (A3) in Theorem \ref{theoremconv} are fulfilled and the claimed norm convergence is proved.
	The proof is complete.
\end{proof}


\section{Convergence of  N\"orlund Means in Vilenkin-Lebesgue points}\label{s3,7}
\markright{N\"orlund  Means and Vilenkin-Lebesgue points}

Our first main result concerning convergence of N\"orlund means reads:

\begin{theorem}\label{Corollary3nnconvn} a) Let $p\geq 1$ and $\{q_k:k\in \mathbb{N}\}$ be a sequence of non-decreasing numbers. 
	
	If the function $f\in L^1(G_m)$ is continuous at a point $x,$ then
	
	$${{t}_{n}}f(x)\to f(x), \ \ \ \text{as}\ \  \ n\to\infty.$$
	Furthermore,
	
	\begin{equation*}
	\underset{n\rightarrow \infty }{\lim }t_nf(x)=f(x)
	\end{equation*}
	for all Vilenkin-Lebesgue points of $f\in L^p(G_m)$.
	
	b)	Let $p\geq 1$ and $\{q_k:k\in \mathbb{N}\}$ be a sequence of non-increasing numbers satisfying the condition \eqref{fn0}. 
	
	If the function $f\in L^1(G_m)$ is continuous at a point $x,$ then
	
	$${{t}_{n}}f(x)\to f(x), \ \ \ \text{as}\ \  \ n\to\infty.$$ 
	Moreover,
	\begin{equation*}
	\underset{n\rightarrow \infty }{\lim }t_nf(x)=f(x)
	\end{equation*}
	for all Vilenkin-Lebesgue points of $f\in L^p(G_m)$.
\end{theorem}

\begin{proof} Let $\{q_k:k\in \mathbb{N}\}$ be a non-decreasing sequence. Suppose that  $x$ is either a point of continuity of a function $f\in L^p(G_m)$ or Vilenkin-Lebesgue point of the function $f\in L^p(G_m).$ According to Proposition \ref{Pointwize} and Proposition \ref{villebfej} we can conclude that
	$$
	\underset{n\rightarrow \infty }{\lim }\vert\sigma_nf(x)-f(x)\vert=0.
	$$
	Hence, by combining \eqref{2b} and \eqref{2bbb} we can conclude that
	\begin{eqnarray*}
		&&\vert t_nf(x)-f(x)\vert \\
		&\leq&\frac{1}{Q_n}\left(\overset{n-2}{\underset{j=1}{\sum}}\left(q_{n-j}-q_{n-j-1}\right)j\vert\sigma_jf(x)-f(x)\vert+q_0n\vert\sigma_nf(x)-f(x)\vert\right)\\
		&\leq&\frac{1}{Q_n}\overset{n-2}{\underset{j=0}{\sum}}\left(q_{n-j}-q_{n-j-1}\right)j\alpha_j+\frac{q_0n\alpha_n}{Q_n}
		:=I+II, 
	\end{eqnarray*}
	where  $  \alpha_n\to 0, \  \text{as} \  n\to\infty.$
	
	Since the sequence $\{q_k:k\in \mathbb{N}\}$ is non-decreasing, we can conclude that	
	
	$$II\leq \alpha_n\to 0, \ \ \text{as} \ \ n\to\infty.$$
	
	On the other hand, since $\alpha_n$  converges to $0,$  we get that there exists an absolute constant $A,$ such that $\alpha_n\leq A$ for any $n\in \mathbb{N}$ and for any $\varepsilon>0$ there exists $N_0\in \mathbb{N},$ such that 
	
	$$\alpha_n< \varepsilon \ \ \text{when } \ \ n>N_0.$$ 
	Hence, 
	\begin{eqnarray*}
		I&=&\frac{1}{Q_n}\overset{N_0}{\underset{j=1}{\sum}}\left(q_{n-j}-q_{n-j-1}\right)j\alpha_j
		+\frac{1}{Q_n}\overset{n-1}{\underset{j=N_0+1}{\sum}}\left(q_{n-j}-q_{n-j-1}\right)j\alpha_j\\
		&:=&I_1+I_2.	
	\end{eqnarray*}
	Since  	
	
	$$\vert q_{n-j}-q_{n-j-1}\vert<2q_{n-1} \ \ \ \text{ and } \ \ \ \alpha_n<A,$$ 
	we obtain that
	
	\begin{eqnarray*}
		I_1=\frac{1}{Q_n}\overset{N_0}{\underset{j=1}{\sum}}\left(q_{n-j}-q_{n-j-1}\right)j\alpha_j \leq \frac{2AN_0q_{n-1}}{Q_n}\to 0, \ \ \ \text{as} \ \ \ n\to \infty
	\end{eqnarray*}
	and
	
	\begin{eqnarray*}
		I_2&=&\frac{1}{Q_n}\overset{n-1}{\underset{j=N_0+1}{\sum}}\left(q_{n-j}-q_{n-j-1}\right)j\alpha_j \\
		&\leq& \frac{\varepsilon}{Q_n}\overset{n-1}{\underset{j=N_0+1}{\sum}}\left(q_{n-j}-q_{n-j-1}\right)j\\
		&\leq& \frac{\varepsilon}{Q_n}\overset{n-1}{\underset{j=0}{\sum}}\left(q_{n-j}-q_{n-j-1}\right)j<\varepsilon.
	\end{eqnarray*}
	We conclude that also $I_2\to 0$ so a) is proved.
	
	Assume now that the sequence is non-increasing and satisfying condition \eqref{fn0}. 
	To prove  convergence in Vilenkin-Lebesgue points we use the estimations \eqref{2b} and \eqref{2bbb} to obtain that
	
	\begin{eqnarray*}
		\vert t_nf-f(x)\vert &\leq&\frac{1}{Q_n}\overset{n-2}{\underset{j=0}{\sum}}\left(q_{n-j-1}-q_{n-j}\right)j\alpha_j+\frac{q_0n\alpha_n}{Q_n} \\
		&:=&III+IV, 
	\end{eqnarray*}
	where   $\alpha_n\to 0, \  \text{as} \  n\to\infty.$
	
	It is evident that
	
	$$IV\leq\frac{q_{0}n\alpha_n}{Q_n}\leq C \alpha_n\to 0, \ \ \text{as} \ \ n\to\infty.$$	
	Moreover, for any $\varepsilon>0$ there exists $N_0\in \mathbb{N},$ such that $\alpha_n< \varepsilon$ when $n>N_0.$ It follows that
	
	\begin{eqnarray*}
		&&\frac{1}{Q_n}\overset{n-2}{\underset{j=1}{\sum}}\left(q_{n-j-1}-q_{n-j}\right)j\alpha_j \\
		&=&\frac{1}{Q_n}\overset{N_0}{\underset{j=1}{\sum}}\left(q_{n-j-1}-q_{n-j}\right)j\alpha_j
		+\frac{1}{Q_n}\overset{n-2}{\underset{j=N_0+1}{\sum}}\left(q_{n-j-1}-q_{n-j}\right)j\alpha_j\\
		&:=&III_1+III_2.
	\end{eqnarray*}
	
	Since sequence is non-increasing, we can conclude that 
	$$\vert q_{n-j}-q_{n-j-1}\vert<2q_{0}.$$  
	Hence,
	
	\begin{eqnarray*}
		III_1\leq\frac{2q_{0}N_0}{Q_n}\to 0, \ \ \ \text{as} \ \ \ n\to \infty
	\end{eqnarray*}
	and
	
	\begin{eqnarray*}
		III_2&\leq&\frac{1}{Q_n}\overset{n-2}{\underset{j=N_0+1}{\sum}}\left(q_{n-j-1}-q_{n-j}\right)j\alpha_j\\
		&\leq&\frac{\varepsilon(n-1)}{Q_n}\overset{n-2}{\underset{j=N_0+1}{\sum}}\left(q_{n-j}-q_{n-j-1}\right) \\
		&\leq& \frac{\varepsilon(n-1)}{Q_n}\left(q_{0}-q_{n-N_0}\right)\\
		&\leq& \frac{2q_0\varepsilon(n-1)}{Q_n}<C\varepsilon.
	\end{eqnarray*}
	Hence, also $III\to 0$ so the proof of part b)  is also complete.
\end{proof}
\begin{corollary} \label{3.7.2t}
	Let $f\in L^p(G_m),$ where  $p\geq 1.$ Then, for all Lebesgue points of $ f\in L^p(G_m),$
	
	\begin{eqnarray*}
		\sigma_{n}f &\rightarrow& f,\text{ \  as \ }n\rightarrow \infty,\\  
		B_{n}f &\rightarrow& f,\text{  \ as \ }n\rightarrow
		\infty, \\
		\beta^{\alpha}_{n}f &\rightarrow& f,\text{  \ as \ }n\rightarrow
		\infty.
	\end{eqnarray*}
\end{corollary}

\begin{theorem}\label{Corollaryconv4} Let $p\geq 1$ and $\{q_k:k\in \mathbb{N}\}$ be a sequence of non-increasing numbers. Then
	
	\begin{equation*}
	\underset{n\rightarrow \infty }{\lim }t_{M_n}f(x)=f(x)
	\end{equation*}
	for all Lebesgue points of $f\in L^p(G_m)$.
\end{theorem}
\begin{proof} By using  Lemma \ref{lemma0nnT121}  we get that
	
	\begin{eqnarray*}
		t_{M_n}f\left(x\right)&=&\underset{G_m}{\int}f\left(t\right)F_n\left(x-t\right) d\mu\left(t\right)\\
		&=&\underset{G_m}{\int}f\left(t\right)D_{M_n}\left(x-t\right)d\mu\left(t\right)
		-\underset{G_m}{\int}f\left(t\right)\psi_{M_n-1}(x-t)\overline{F^{-1}}_{M_n}(x-t)\\
		&:=&I-II.
	\end{eqnarray*}
	By applying Proposition \ref{Snae1} we get that $I=S_{M_n}f(x)\to f(x)$
	for all Lebesgue points of $f\in L^p(G_m)$, where $p\geq 1.$
	
	Moreover, according to Proposition \ref{vilprop} we find that
	
	$$\psi_{M_n-1}(x-t)=\psi_{M_n-1}(x)\overline{\psi}_{M_n-1}(t)$$ 
	and
	
	$$II=\psi_{M_n-1}(x)\underset{G_m}{\int}f\left(t\right)\overline{F^{-1}}_{M_n}(x-t)\overline{\psi}_{M_n-1}(t)d(t).$$
	By combining Theorem \ref{Theorem 1.1.6}  and Corollary \ref{lemma0nnT12} we find that the  function 
	
	$$f\left(t\right)\overline{F^{-1}}_{M_n}(x-t)\in L^p(G_m) \ \ \text{ where} \ \  p\geq  1 \ \ \text{for any } \ \ x\in G_m, $$
	and $II$ are Fourier coefficients of an integrable function. Hence, according to the Riemann-Lebesgue Lemma we get that
	$$II\to 0,\text{  \ as \ }n\rightarrow
	\infty, \  \text{for any } \  x\in G_m.$$
	
	The proof is complete.
\end{proof}

\begin{corollary}\label{3.7.4t}
	Let $f\in L^p(G_m),$ where  $p\geq 1.$ Then, for all Lebesgue points of $ f\in L^p(G_m),$
	
	\begin{eqnarray*}
		\sigma^{\alpha}_{M_n}f &\rightarrow& f,\text{  \ as \ }n\rightarrow
		\infty,  \\
		V^{\alpha}_{M_n}f &\rightarrow& f,\text{  \ as \ }n\rightarrow
		\infty, \\
		U^{\alpha}_{M_n}f &\rightarrow& f,\text{  \ as \ }n\rightarrow
		\infty. \\
	\end{eqnarray*}
\end{corollary}

\nocite{*}
\def\refname{References}
\def\printchapternonum{}
\bibliographystyle{abbrv}
\bibliography{bib_source/references}
\cleardoublepage

%
%
%
%
%
%
%
%
%

\end{document}